\documentclass[12pt]{amsart}
\usepackage{amsfonts,amsthm,latexsym,amsmath,amssymb,amscd,amsmath, mathrsfs,
url}
\newcounter{count}

\numberwithin{count}{section}
\newtheorem{Lemma}[count]{Lemma}

\newtheorem{Theorem}[count]{Theorem}
\newtheorem{Statement}[count]{Statement}

\begin{document}

\author[T. H.~Nguyen]{Thu Hien Nguyen}

\address{Partially supported by the Akhiezer Foundation. Department of Mathematics \& Computer Sciences, V. N. Karazin Kharkiv National University,
4 Svobody Sq., Kharkiv, 61022, Ukraine}
\email{nguyen.hisha@karazin.ua}

\author[A.~Vishnyakova]{Anna Vishnyakova}
\address{Department of Mathematics \& Computer Sciences, V. N. Karazin Kharkiv National University,
4 Svobody Sq., Kharkiv, 61022, Ukraine}
\email{anna.m.vishnyakova@univer.kharkov.ua}

\title[Entire functions of the Laguerre-P\'olya class]
{On the closest to zero roots and the second quotients of Taylor coefficients 
of entire functions from the Laguerre-P\'olya I class }

\begin{abstract}

For an entire function $f(z) = \sum_{k=0}^\infty a_k z^k, a_k>0,$  we show that if $f$
belongs to the Laguerre-P\'olya class,  and the quotients   $q_k := \frac{a_{k-1}^2}{a_{k-2}a_k}, 
k=2, 3, \ldots $ satisfy the condition $q_2 \leq q_3,$ then $f$ has at least one zero in the segment 
$[-\frac{a_1}{a_2},0].$  We also give necessary conditions and sufficient conditions of the 
existence of  such a zero in terms of the quotients $q_k$ for $k=2,3, 4.$

\end{abstract}



\maketitle

\section{Introduction}

The zero distribution of  entire functions, its sections and tails have been 
studied by  many authors, see, for example, the remarkable
survey of the topic in \cite{iv}.  In this paper we investigate new 
conditions under which some special entire functions have only real zeros. 
First, we give the definition of the well-known Laguerre-P\'olya class.

{\bf Definition 1}.  A real entire function $f$ is said to be in the {\it
Laguerre-P\'olya class}, written $f \in \mathcal{L-P}$, if it can
be expressed in the form
\begin{equation}
\label{lpc}
 f(x) = c x^n e^{-\alpha x^2+\beta x}\prod_{k=1}^\infty
\left(1-\frac {x}{x_k} \right)e^{xx_k^{-1}},
\end{equation}
where $c, \alpha, \beta, x_k \in  \mathbb{R}$, $x_k\ne 0$,  $\alpha \ge 0$,
$n$ is a nonnegative integer and $\sum_{k=1}^\infty x_k^{-2} <
\infty$. As usual, the product on the right-hand side can be
finite or empty (in the latter case the product equals 1).

This class is essential in the theory of entire functions due to the fact that  
the polynomials with only real zeros converge locally uniformly to these and 
only these functions. The following  prominent theorem 
states an even  stronger fact. 

{\bf Theorem A} (E.Laguerre and G.P\'{o}lya, see, for example,
\cite[p. 42-46]{HW}). {\it   

(i) Let $(P_n)_{n=1}^{\infty},\  P_n(0)=1, $ be a sequence
of complex polynomials having only real zeros which  converges uniformly in
the circle $|z|\le A, A > 0.$ Then this sequence converges
locally uniformly to an entire function
 from the $\mathcal{L-P}$ class.
 
(ii) For any $f \in \mathcal{L-P}$ there is a
sequence of complex polynomials with only real zeros which
converges locally uniformly to $f$.}

In this paper, we study the functions from the important subclass of the class $\mathcal{L-P}.$

{\bf Definition 2}.  A real entire function $f$ is said to be in the {\it Laguerre-
P\'olya class of type I}, 
written $f \in \mathcal{L-P} I$, if it can
be expressed in the following form  
\begin{equation}  \label{lpc1}
 f(x) = c x^n e^{\beta x}\prod_{k=1}^\infty
\left(1+\frac {x}{x_k} \right),
\end{equation}
where $c \in  \mathbb{R},  \beta \geq 0, x_k >0 $, 
$n$ is a nonnegative integer,  and $\sum_{k=1}^\infty x_k^{-1} <
\infty$. 

The famous theorem by E.~Laguerre and G.~P\'olya   (see, for example, 
\cite[chapter VIII, \S 3]{lev}) states that the polynomials with only real 
nonpositive zeros converge locally uniformly to the 
function from the class $\mathcal{L-P}I.$ The following theorem 
states a  stronger fact. 

{\bf Theorem B} (E.Laguerre and G.P\'{o}lya, see, for example,
\cite[chapter VIII, \S 3]{lev}). {\it 
 
(i) Let $(P_n)_{n=1}^{\infty},\  P_n(0)=1, $ be a sequence
of complex polynomials having only real negative zeros which  
converges uniformly in the circle  $|z|\le A, A > 0.$ Then this 
sequence converges locally uniformly to an entire function
 from the class $\mathcal{L-P}I.$
 
(ii) For any $f \in \mathcal{L-P}I$ there is a
sequence of complex polynomials with only real nonpositive 
zeros which  converges locally uniformly to $f$.}

For various properties and characterizations of the
Laguerre-P\'olya   class and the Laguerre-P\'olya class of type I,  see 
\cite[p. 100]{pol}, \cite{polsch}  or  \cite[Kapitel II]{O}.

Note that for a real entire function (not identically zero) of order less than $2$ having 
only real zeros is  equivalent to belonging to the Laguerre-P\'olya class. The situation is 
different when an entire function is of order $2$. For example, the function 
$f_1(x)= e^{- x^2}$ belongs to the Laguerre-P\'olya class, but the function 
$f_2(x)= e^{x^2}$ does not.

Let  $f(z) = \sum_{k=0}^\infty a_k z^k$  be an entire function with positive 
coefficients.  We define the quotients $p_n$ and $q_n$:

\begin{eqnarray}
\label{qqq} &  p_n=p_n (f):=\frac {a_{n-1}}{a_n},\ n\geq 1;
\\
\nonumber &
q_n=q_n(f):=\frac {p_n}{p_{n-1}}=\frac {a_{n-1}^2}{a_{n-2}a_n},\
n\geq 2.
\end{eqnarray}
The following formulas can be verified by straightforward calculation.
\begin{eqnarray}
\label{defq}
& a_n=\frac {a_0}{p_1 p_2 \ldots p_n},\ n\geq 1\ ; \\
\nonumber & a_n=\frac
{a_1}{q_2^{n-1} q_3^{n-2} \ldots q_{n-1}^2 q_n} \left(
\frac{a_1}{a_0} \right) ^{n-1},\ n\geq 2.
\end{eqnarray}

It is not trivial to understand whether a given entire function has only real zeros. 
In 1926, J. I. Hutchinson found the following 
sufficient condition for an entire function with positive coefficients to 
have only real zeros.

{\bf Theorem C} (J. I. Hutchinson, \cite{hut}). { \it Let $f(z)=
\sum_{k=0}^\infty a_k z^k$, $a_k > 0$ for all $k$. 
Then $q_n(f)\geq 4$, for all $n\geq 2,$  
if and only if the following two conditions are fulfilled:\\
(i) The zeros of $f(z)$ are all real, simple and negative, and \\
(ii) the zeros of any polynomial $\sum_{k=m}^n a_kz^k$, $m < n,$  formed 
by taking any number 
of consecutive terms of $f(z) $, are all real and non-positive.}

For some extensions of Hutchinson's results see,
for example, \cite[\S 4]{cc1}. 

The following function has a prominent role.

{\bf Definition 3.} The entire function $g_a(z) =\sum _{j=0}^{\infty} z^j a^{-j^2}$, $a>1,$
is called the partial theta-function.

Note that its second quotients are constant: $q_n(g_a) = a^2,$ for every $n \in \mathbb{N}.$
The  survey \cite{War} by S.O.~Warnaar contains the history of
investigation of the partial theta-function and its interesting properties.

{\bf Theorem D} (O. Katkova, T. Lobova, A. Vishnyakova, \cite{klv}).  {\it There exists a 
constant $q_\infty $
$(q_\infty\approx 3{.}23363666 \ldots) $ such that:
\begin{enumerate}
\item
$g_a(z) \in \mathcal{L-P} \Leftrightarrow \ a^2\geq q_\infty ;$
\item
$g_a(z) \in \mathcal{L-P} \Leftrightarrow \ $  there exists $x_0 \in (- a^3, -a)$ 
such that $ \  g_a(x_0) \leq 0;$
\item
for a given $n\geq 2$ we have $S_{n}(z,g_a) \in \mathcal{L-P}$ $ \ \Leftrightarrow \ $
there exists $x_n \in (- a^3, -a)$ such that $ \ S_n(x_n,g_a) \leq 0;$
\item
$ 4 = c_2 > c_4 > c_6 > \ldots $  and    $\lim_{n\to\infty} c_{2n} = q_\infty ;$
\item
$ 3= c_3 < c_5 < c_7 < \ldots $  and    $\lim_{n\to\infty} c_{2n+1} = q_\infty.$
\end{enumerate}}
There is a series of works by V.P. Kostov dedicated to the interesting properties of zeros 
of  the partial theta-function and its  derivative (see \cite{kos0}, \cite{kos1}, \cite{kos2}, 
\cite{kos3}, \cite{kos03}, \cite{kos04}, \cite{kos4}, \cite{kos5} and \cite{kos5.0}). 

A wonderful paper \cite{kosshap} among the other results explains the role 
of the constant $q_\infty $  in the study of the set of entire functions with positive 
coefficients having all Taylor truncations with only real zeros. 

Let us consider the entire function $f(z) = e^z = \sum_{k=0}^\infty \frac{z^k}{k!} = 
1 + z + \frac{z^2}{2!} + \frac{z^3}{3!} + \frac{z^4}{4!} + \frac{z^5}{5!} + \ldots,$ 
which belongs to the Laguerre-P\'olya class of type I. We can observe that its second quotients are
$q_k(f) = \frac{a_{k-1}^2}{a_{k-2}a_k} = \frac{k}{k-1}, k \geq 2.$

The following statement is the analogue (for entire functions) of the Newton inequalities which 
are necessary conditions for real polynomials with positive coefficients to have only real zeros
(or, equivalently, to belong to the Laguerre-P\'olya class). This fact is well-known to experts, but 
it is a kind of folklore: it is easier to give the proof than to find an appropriate reference. 

\begin{Statement}
\label{st1}
Let $f(z)=\sum_{k=0}^\infty a_k z^k$, $a_k > 0$ for all $k,$ be an entire function 
from the Laguerre-P\'olya class of type I. Then $q_n(f) \geq \frac{n}{n-1}$, for all 
$n \geq 2.$ Moreover, if there exists $m = 2, 3, \ldots,$ such that $q_m(f) = \frac{m}{m-1},$ 
then $f(z) = ce^{\alpha z}, c>0, \alpha>0.$  
\end{Statement}

The following statement is the analogue of Lemma 2.1 from the work \cite{ngthv2}. For 
the reader's convenience, we further provide the proof of this lemma.

\begin{Lemma}
\label{lem0}
If $f(z) = \sum_{k=0}^\infty a_kz^k, a_k>0,$ belongs to $\mathcal{L-P}I,$ 
then $q_3(q_2 - 4) + 3 \geq 0.$
In particular, if $q_3 \geq q_2,$ then $q_2 \geq 3.$
\end{Lemma}

In \cite{klv1}, some  entire functions  with a convergent sequence of second quotients 
of coefficients  are investigated. The main question of \cite{klv1} is whether a function 
and its Taylor sections belong to the Laguerre-P\'olya class. In \cite{BohVish} and 
\cite{Boh}, some important special functions with increasing sequence of second 
quotients of Taylor coefficients  are studied. 

In \cite{ngthv1} and \cite{ngthv2}, we have found the conditions for entire 
functions of order zero with monotonic second quotients to belong to the 
Laguerre-P\'olya class.

{\bf Theorem  E} (T. H. Nguyen, A. Vishnyakova, \cite{ngthv1}).
{\it Let $f(z)=\sum_{k=0}^\infty a_k z^k $, $a_k > 0$ for all $k$, be an 
entire function.  Suppose that $q_n(f)$ are decreasing in $n$, i.e.  $q_2 \geq q_3 \geq 
q_4 \geq \ldots, $  and  $\lim\limits_{n\to \infty} q_n(f) = b \geq q_\infty$. 
Then all the zeros of $f$ are  real and negative, or in other words,  $f \in \mathcal{L-P}$.}

It is easy to see that, if only the estimation of $q_n(f)$ from below is given
and the assumption of monotonicity is omitted, then the constant $4$ in 
$q_n (f) \geq 4 $ is the smallest possible to conclude that $f \in \mathcal{L-P}$. 

We have also investigated the case when $q_n(f)$ are increasing in $n$ and have 
obtained the necessary condition.

{\bf Theorem F} (T. H. Nguyen, A. Vishnyakova, \cite{ngthv2}).
{\it Let $f(z)=\sum_{k=0}^\infty a_k z^k $, $a_k > 0$ for all $k,$  be an 
entire function.  Suppose that the quotients $q_n(f)$ are increasing in $n$, 
and  $\lim\limits_{n\to \infty} q_n(f) = c < q_\infty$. 
Then the function $f$ does not belong to the  Laguerre-P\'olya class.}

The first author studied the  function $f_a(z) = \sum_{k=0}^\infty \frac{z^k}{(a^k+1)(a^{k-1}+1)
\cdot \ldots \cdot (a+1)},$ $a>1,$  which is relative to the partial theta-function and the Euler function 
and found the conditions for it to belong to the Laguerre-Pol\'ya class (see \cite{ngth1}).

It turns out that for many important entire functions with positive
coefficients $f(z)=\sum_{k=0}^\infty a_k z^k $ (for example, partial theta-function
from \cite{klv}, functions from \cite{BohVish} and \cite{Boh}, function 
$f_a(z) = \sum_{k=0}^\infty \frac{z^k}{(a^k+1)(a^{k-1}+1)
\cdot \ldots \cdot (a+1)},$ $a>1,$ and others) the following two conditions
are equivalent: 

(i) $f$ belongs to the Laguerre-Pol\'ya class of type I,  

and 

(ii) there exists $z_0 \in [-\frac{a_1}{a_2},0]$ such that $f(z_0) \leq 0.$

The following theorem is a new necessary condition for an entire function to belong to the 
Laguerre-Pol\'ya class of type I in terms of the closest to zero roots.

\begin{Theorem}
\label{th:mthm1}
Let $f(z)=\sum_{k=0}^\infty a_k z^k $, $a_k > 0$ for all $k,$  be an 
entire function.  Suppose that the quotients $q_n(f)$ satisfy the following condition: 
$q_2(f) \leq q_3(f).$ If the function $f$ belongs to the  Laguerre-P\'olya class, then 
there exists  $z_0 \in [-\frac{a_1}{a_2},0]$ such that $f(z_0) \leq 0$. 
\end{Theorem}

The following theorem gives  a necessary condition for an entire function to belong to the 
Laguerre-Pol\'ya class of type I in terms of the second quotients of its Taylor coefficients.

\begin{Theorem}
\label{th:mthm2}
If $f(z) = \sum_{k=0}^\infty  a_k z^k,$ $a_k > 0$ for all $k,$   
belongs to the Laguerre-P\'olya class, $q_2(f) < 4$ and $q_2(f) \leq q_3(f),$  then 
$$q_3(f) \leq \frac{- q_2(f)(2q_2(f)-9)+2(q_2(f)-3)\sqrt{q_2(f)(q_2(f)-3)}}{q_2(f)(4-q_2(f))}.$$ 
\end{Theorem}

In the proof of Theorem \ref{th:mthm1}, using the Hutchinson's idea, we show 
that if $q_2(f) \geq 4$ (and $q_j >1, j=3, 4, \ldots$), then there exists a point  
$z_0 \in [-\frac{a_1}{a_2},0]$ such that $f(z_0) \leq 0$.  We  present the sufficient 
condition for the existence of such a point $z_0 $  for the case $q_2(f) < 4.$

\begin{Theorem}
\label{th:mthm3}
Let $f(z)=\sum_{k=0}^\infty a_k z^k $, $a_k > 0$ for all $k,$  be an 
entire function and  $3  \leq  q_2(f) < 4, q_3(f) \geq 2,$ and $q_4(f) \geq 3.$ If $q_3(f) 
\leq \frac{8}{d(4-d)},$ where 
$d = \min(q_2(f), q_4(f)),$ then  there exists $z_0 \in [-\frac{a_1}{a_2},0]$ 
such that $f(z_0) \leq 0$.  
\end{Theorem}

\section{Proof of Statement \ref{st1}}

\begin{proof}
We give the proof by induction on $k.$
{\it Base case}:  $k=2.$ If $f$ does 
not have any real roots, then $f(z) = ce^{\alpha z},$ i.e. the statement is fulfilled (see (\ref{lpc1})). If $f$ has 
at least one real root, we denote by $\{z_k \}_{k=1}^\alpha, \   \alpha \in \mathbb{N}\cup \{ \infty \}$ 
the set of roots of $f$. Then $0 < - \sum_{k=1}^\alpha \frac{1}{z_k} < \infty$, which follows that 
$\sum_{k=1}^\alpha \frac{1}{z_k^2} < \infty,$ whence $\sum_{k=1}^\alpha \frac{1}{z_k^2} = 
a_1^2 - 2a_0a_2 > 0.$ Consequently, $q_2(f) \geq 2,$ and if $q_2 = 2,$ then $f(z) = ce^{\alpha z}.$

{\it Inductive step}: Suppose that the statement is true for $k-1,$   $k = 3, 4, \ldots$. 
Obviously, if  $f \in \mathcal{L-P}I,$ then $f^{(s)} \in \mathcal{L-P}I,$ for any $s \in \mathbb{N}.$ 
So, $f^{(k-2)}(z) = a_{k-2} (k-2)! + a_{k-1} (k-1)! z + a_k \frac{k!}{2!}z^2 + \ldots$ $\in \mathcal{L-P}I.$ 
Then, $q_2(f^{(k-2)}) =$ $\frac{a_{k-1}^2}{a_{k-2}a_k} \frac{2((k-1)!)^2}{(k-2)!k!} = $
$\frac{a_{k-1}^2}{a_{k-2}a_k} \frac{2(k-1)}{k} \geq 2,$ whence 
$q_k(f) \geq \frac{k}{k-1}.$
If $q_k(f) = \frac{k}{k-1},$ then $q_{k-1}(f^{\prime}) =\frac{k-1}{k-2},$ and, by the
induction conjecture, $f^{\prime}(z) = c e^{\alpha z}.$   So, $f(z) = \frac{c}{\alpha} e^{\alpha z} 
+\lambda,$ where $\lambda$ is a constant. Since $f \in \mathcal{L-P}I,$ 
then $\lambda = 0,$ and we are done.
\end{proof}

\section{Proof of Lemma \ref{lem0}}
\begin{proof}
Suppose that  $f \in \mathcal{L-P}I,$ and denote by  $0 < z_1 \leq z_2 \leq z_3 \leq  \ldots $  
the real roots of $f$. 

According to the Cauchy-Bunyakovsky-Schwarz inequality, we obtain
$$(\frac{1}{z_1} + \frac{1}{z_2} + \ldots )(\frac{1}{z_1^3} + \frac{1}{z_2^3} + 
\ldots) \geq (\frac{1}{z_1^2 }+ \frac{1}{z_2^2} + \ldots)^2.$$

By the Vieta's formulas, we have $ \sigma_1:=  \sum_{k=1}^\infty \frac{1}{z_k} =  - \frac{a_1}{a_0},$ 
$\sigma_2 = \sum_{1<i<j<\infty} \frac{1}{z_iz_j}$ $= \frac{a_2}{a_0},$ and
$\sigma_3 =  \sum_{1<i<j<k<\infty} \frac{1}{z_iz_jz_k}
= - \frac{a_3}{a_0}.$ 
By the  identities:
$\sum_{k=1}^\infty \frac{1}{z_1^2} = \sigma_1^2 - 2\sigma_2,$ and 
$\sum_{k=1}^\infty \frac{1}{z_1^3} = \sigma_1^3 - 3\sigma_1\sigma_2 + 3\sigma_3,$ 
 we have 
$$\sigma_1 (\sigma_1^3 - 3\sigma_1\sigma_2 + 3\sigma_3) \geq (\sigma_1^2 - 2\sigma_2)^2,$$
or
$$\frac{a_1^2a_2}{a_0^3} + 3\frac{a_1a_3}{a_0^2} - 4\frac{a_2^2}{a_0^2} = 
\frac{a_1a_3}{a_0^2} \left(\frac{a_1a_2}{a_3a_0}  -4 \frac{a_2^2}{a_1 a_3} +3  \right) \geq 0,  $$
which is equivalent to
$$q_3(q_2 - 4) + 3 \geq 0.$$
In particular, when $q_2 \leq q_3,$ then either $q_2 \geq 4, $ or $q_2 < 4,$ and, from the 
previous inequality, we have
$$q_2(q_2 - 4) + 3 \geq 0.$$
Therefore, in both cases, we get $q_2 \geq 3.$

\end{proof}

\section{Proof of Theorem \ref{th:mthm1}}

Without loss of generality, we can assume that $a_0=a_1=1,$ since we can 
consider a function $g(x) =a_0^{-1} f (a_0 a_1^{-1}x) $  instead of 
$f(x),$ due to the fact that such rescaling of $f$ preserves its property of 
having real zeros and preserves the second quotients:  $q_n(g) =q_n(f)$ 
for all $n.$ During the proof we use notation $p_n$ and $q_n$ instead of 
$p_n(f)$ and $q_n(f).$  

We will consider a function  $$\varphi(x) = f(-x) = 1 - x + \sum_{k=2}^\infty  \frac{ (-1)^k x^k}
{q_2^{k-1} q_3^{k-2} \ldots q_{k-1}^2 q_k}$$  instead of $f.$

Let us introduce some more notations. For an entire function $\varphi$, by $S_n(x, \varphi)$ 
and $R_{n}(x, \varphi)$ we denote the $n$th partial sum and the $n$th remainder of the series,
i.e. 
$$S_n(x, \varphi) = \sum_{k=0}^n \frac{(-1)^k x^k}{q_2^{k-1} q_3^{k-2} \ldots q_{k-1}^2 q_k},$$
and 
$$R_{n}(x, \varphi) = \sum_{k=n}^\infty \frac{(-1)^k  x^k}{q_2^{k-1} q_3^{k-2} \ldots q_{k-1}^2 q_k}.$$

At first we consider the simple case when $q_2 \geq 4$ (and $q_j >1, j=3, 4, \ldots$,
by Statement \ref{st1}, since $\varphi$ belongs to the  Laguerre-P\'olya class).
We use the idea of J. I. Hutchinson, see \cite{hut}.
Suppose that   $x \in (1, q_2).$  Then we obtain 
$$ 1 <  x  >  \frac{x^2}{q_2} > \frac{x^3}{q_2^2 q_3} > 
\cdots >  \frac{  x^k}{q_2^{k-1} q_3^{k-2} \ldots q_{k-1}^2 q_k} > \cdots .
$$
Thus, for $x \in (1, q_2)$ we have
$$\varphi(x) = \left(1 -x +\frac{x^2}{q_2}\right) - \left( \frac{x^3}{q_2^2q_3} -
\frac{x^4}{q_2^3 q_3^2 q_4}\right)  - $$
$$ \left( \frac{x^5}{q_2^4q_3^3 q_4^2 q_5} -
\frac{x^6}{q_2^5 q_3^4 q_4^3 q_5^2 q_6}\right) - \ldots < 
\left(1 -x +\frac{x^2}{q_2}\right).$$
So, $\varphi(0)=1 >0,$ and for $x_0 = \sqrt{q_2}$ we obtain
$$\varphi(x_0) < \left(1 - \sqrt{q_2} +1\right) \leq 0.$$
So, $\varphi$ has a zero $x_0\in (0; q_2).$ Thus, for $q_2 \geq 4$ 
Theorem \ref{th:mthm1}  is proved.

Now we consider the case $q_2 <4$ and $q_2 \leq q_3.$ By Lemma~\ref{lem0}, 
if $q_2 \leq q_3,$ then $q_2 \geq 3.$ So, we suppose that  $3\leq q_2 <4, q_3 \geq q_2$ 
and $q_4 \geq \frac{4}{3}$
($q_4 \geq \frac{4}{3}$ by Statement \ref{st1}, since $\varphi$ belongs to the  
Laguerre-P\'olya class).

The following lemma plays a key role  in the proof of Theorem \ref{th:mthm1}.
It is a generalization of Lemma 2.3 from the paper \cite{ngthv2}. The first variant
of this lemma for the simplest case $a=b=c \geq 3$ was proved in \cite{klv} (see
Lemma 1 in \cite{klv}).

\begin{Lemma}
\label{th:lm1}
Let $P(z) = 1 - z + \frac{z^2}{a} - \frac{z^3}{a^2b} + \frac{z^4}{a^3b^2c}$ be a polynomial, 
$3 \leq  a < 4,$ $b \geq a,$ and  $c\geq 4/3.$ 
Then $$\min_{0 \leq \theta \leq 2\pi}|P(ae^{i \theta})| \geq \frac{a}{b^2c}.$$
\end{Lemma}

\begin{proof} By direct calculation, we have
$$|P(ae^{i\theta})|^2 =
 (1 - a\cos{\theta} + a\cos{2\theta} - \frac{a}{b}\cos{3\theta} + \frac{a}{b^2c}\cos{4\theta})^2 +$$
$$ (- a\sin{\theta} + a\sin{2\theta} - \frac{a}{b}\sin{3\theta} + \frac{a}{b^2c}\sin{4\theta})^2$$
$$=1 + 2a^2 + \frac{a^2}{b^2} + \frac{a^2}{b^4c^2} - 2a\cos{\theta} + 2a\cos{2\theta} - 2\frac{a}{b}\cos{3\theta}$$
$$+2\frac{a}{b^2c}\cos{4\theta} - 2a^2\cos{\theta} + 2\frac{a^2}{b}\cos{2\theta} - 2\frac{a^2}{b^2c}\cos{3\theta}$$
$$- 2\frac{a^2}{b}\cos{\theta} + 2\frac{a^2}{b^2c}\cos{2\theta} - 2\frac{a^2}{b^3c}\cos{\theta}.$$

Set $t := \cos{\theta}, t \in [-1,1].$ Since $\cos{2\theta} =  2t^2 - 1,$
$\cos{3\theta} = 4t^3 - 3t,$ and $\cos{4\theta} =   8t^4 - 8t^2 +1,$ we get

$$|P(ae^{i\theta})|^2 = \frac{16a}{b^2c}t^4 + \left(-\frac{8a}{b} - \frac{8a^2}{b^2c}\right)t^3 + 
\left(4a - \frac{16a}{b^2c} + 
\frac{4a^2}{b} + \frac{4a^2}{b^2c}\right)t^2 +$$
$$ \left(-2a + \frac{6a}{b} - 2a^2 + \frac{6a^2}{b^2c} - \frac{2a^2}{b} - \frac{2a^2}{b^3c}\right)t $$
$$ + \left(1 + 2a^2 + \frac{a^2}{b^2} + \frac{a^2}{b^4c^2} - 2a + \frac{2a}{b^2c} - 
\frac{2a^2}{b} - \frac{2a^2}{b^2c}\right).$$

We want to show that $\min_{0 \leq \theta \leq 2\pi} |P(ae^{i\theta})|^2 \geq \frac{a^2}{b^4c^2},$
or  to prove the inequality $\min_{0 \leq \theta \leq 2\pi} |P(ae^{i\theta})|^2 - \frac{a^2}{b^4c^2} \geq 0.$
Using the last expression, we see that the inequality we want to get is equivalent to the following: for
 all  $t \in [-1, 1]$
the next inequality holds
$$\frac{16a}{b^2c}t^4 - \frac{8a}{b} \big( 1+ \frac{a}{bc} \big)t^3 + 4a \big( 1 - \frac{4}{b^2c} + 
\frac{a}{b} + \frac{a}{b^2c}\big)t^2 - 2a\big( 1 - \frac{3}{b} +a - \frac{3a}{b^2c} + \frac{a}{b} $$
$$+ \frac{a}{b^3c} \big)t + \big( 1 + 2a^2 + \frac{a^2}{b^2} - 2a + \frac{2a}{b^2c} - \frac{2a^2}{b} - 
\frac{2a^2}{b^2c} \big) \geq 0.$$
Set $y := 2t,$  $y \in [-2, 2].$ We rewrite the last inequality in the form
$$\frac{a}{b^2c}y^4 - \frac{a}{b} \left( 1+ \frac{a}{bc} \right)y^3 + a \left( 1 - 
\frac{4}{b^2c} + \frac{a}{b} + \frac{a}{b^2c}\right)y^2 $$
$$- a\left( 1 - \frac{3}{b} +a - \frac{3a}{b^2c} + \frac{a}{b} + \frac{a}{b^3c} \right)y + $$
$$\left( 1 + 2a^2 + \frac{a^2}{b^2} - 2a + \frac{2a}{b^2c} - \frac{2a^2}{b} - \frac{2a^2}{b^2c} \right) \geq 0.$$

We note that the coefficient of $y^4$ is positive, and the coefficient of  $y^3$ is negative.  It is easy to show that the 
other coefficients are also sign-changing. For $y^2$:  $1 - \frac{4}{b^2c} > 0$ since $b^2c > 4$, 
thus, $$1 + \frac{a}{b} + \frac{a}{b^2c} - \frac{4}{b^2c} = (1 - \frac{4}{b^2c}) + \frac{a}{b} + \frac{a}{b^2c} > 0.$$
For $y$:   $$1 + a + \frac{a}{b} + \frac{a}{b^3c} - \frac{3}{b} - \frac{3a}{b^2c} = (1 + a - \frac{3}{b}) + $$
$$(\frac{a}{b} - \frac{3a}{b^2c}) + \frac{a}{b^3c} > 0.$$   Finally, $$1 + 2a^2 + \frac{a^2}{b^2} - 2a - 2\frac{a^2}{b} - 
2\frac{a^2}{b^2c} + 2\frac{a}{b^2c} =$$
$$(1 + a^2 - 2a) + (a^2 - 3\frac{a^2}{b}) + (\frac{a^2}{b} - 2\frac{a^2}{b^2c}) + \frac{a^2}{b^2} + 2\frac{a}{b^2c} > 0,$$ since 
$1 - 2a + a^2  \geq 0;$  $a^2 - 3\frac{a^2}{b} \geq 0 $  and $\frac{a^2}{b} - 2\frac{a^2}{b^2c} > 0,$ by
our assumptions.
 
Consequently, the inequality we need holds for any $y \in [-2, 0]$, and it is sufficient to prove it for $y \in [0, 2]$.
Multiplying our inequality by $\frac{b^2c}{a}$, we get 
$$y^4 - (bc + a)y^3 + (b^2c + abc + a - 4) y^2 - (b^2c + ab^2c + abc + 
\frac{a}{b} - 3bc - 3a)y$$
$$ + (\frac{b^2c}{a} + 2ab^2c + ac - 2b^2c - 2abc - 2a +2) =: \psi(y),$$ 
and we want to prove 
that $\psi (y) \geq 0$ for all $y \in [0, 2].$

First, we consider the case $b \geq a \in [3, 4),  c\geq \frac{5}{3}.$  Let $\chi(y) := \psi(y) - \frac{1}{b}(b - a) y,$ 
whence $\chi(y) \leq \psi(y)$ for all $y \in [0, 2]$.
It is sufficient to prove that $\chi (y) \geq 0$ for all $y \in [0, 2].$ We have  $$\chi(0) = \psi(0) = 
\frac{b^2c}{a} + 2ab^2c + ac - 2b^2c - 2abc - 2a + 2 \geq 0,$$ as it was previously shown.
We also have  $\chi(2) = \psi(2) - \frac{2}{b}(b - a) \geq 0,$ since
$$\psi(2) = -2bc - 2\frac{a}{b} + \frac{b^2c}{a} + ac + 2 = \frac{1}{b} \left(2(b - a) + 
\frac{b^2c}{a}(b - a) \right.$$
$$\left.  - bc(b - a)\right) =\frac{1}{b}(b - a) \left(2 + 
\frac{bc}{a} (b - a)\right) \geq \frac{2}{b} (b - a) \geq 0.$$

Now we consider the following function:  $$\nu(y):=\chi^{\prime\prime}(y) = 
\psi^{\prime\prime}(y) = 12 y^2 - 6 (bc + a)y + 2(b^2c + abc + a - 4).$$
The vertex point of this parabola is $y_v = \frac{bc + a}{4} \geq 2$  for $b \geq a\geq 3,  
c\geq \frac{5}{3}.$ Therefore, 
$$  \nu(y) \geq \nu(2)\  \mbox{for all} \   y\in[0, 2]. $$
We have
$$\nu(2) = 48 - 12(bc +a) +2b^2c  +2abc +2a -8 = 40  - 12 bc - 10a +2b^2c +2abc =   $$
$$ 2bc(a -3) +2bc (b -3) +10(4 - a) \geq 0,  $$
under our assumptions. Thus, $\chi^\prime (y)$ is increasing for $y \in [0, 2].$
We have 
$$\chi^\prime (y)  = 4y^3 -3(bc+a) y^2 +2 (b^2c + abc + a - 4) y - $$
$$  (b^2c + ab^2c + abc + \frac{a}{b} - 3bc - 3a) - \frac{1}{b}(b - a).  $$
We want to prove that $\chi^\prime (y) \leq 0$ for all $y \in [0, 2].$ Since 
$\chi^\prime $ is an increasing function, it is sufficient to show that 
$\chi^\prime (2) \leq 0.$  We have
$$  \chi^\prime (2) =32 -12 (bc+a) +$$
$$4 (b^2c + abc + a - 4) - (b^2c + ab^2c + abc + \frac{a}{b} - 3bc - 3a) - 1+ \frac{a}{b}= $$
$$15 -9bc -5a +3 b^2c +3abc - ab^2c =(15 -5a) - bc(a-3)(b-3) \leq 0, $$
under our assumptions.

Thus, we have proved that $\chi$ is decreasing for $y \in [0, 2].$ Hence, the fact that 
$\chi (y) \geq 0$ for all $y \in [0, 2],$ is equivalent to $\chi (2) \geq 0,$ that was proved
above. We obtain $\psi(y) \geq \chi(y) \geq 0 $ for all $y \in [0, 2].$ So, for the case 
$b \geq a \in [3, 4),  c\geq \frac{5}{3}$ Lemma \ref{th:lm1} is
proved.

It remains to consider the case $3 \leq a <4,  b\geq a, \frac{4}{3} \leq c < \frac{5}{3}.$
We want to prove that $\psi^{\prime\prime}(y) = \chi^{\prime\prime}(y) =  12 y^2 - 6 (bc + a)y + 
2(b^2c + abc + a - 4) \geq 0$ for all $ y\in[0, 2].$ If the vertex point of this parabola  
$y_v = \frac{bc + a}{4}$ is not less than 2, then we have proved that $  \psi^{\prime\prime}(y) 
\geq \psi^{\prime\prime}(2) \geq 0.$ Suppose that $0 <y_v = \frac{bc + a}{4} < 2.$
Then we want to show that $\psi^{\prime\prime}(y_v) \geq 0.$ We have
$$ \psi^{\prime\prime}(y_v) =  12 \left(\frac{bc + a}{4}\right)^2 - 6 (bc + a)\left(\frac{bc + a}{4}\right) + 
2(b^2c + abc + a - 4)  $$
$$ = -\frac{1}{4} \left(3 b^2c^2 -(2ab +8b^2)c + (3a^2 -8a  +32)   \right) =: -\frac{1}{4} h_{a,b}(c).$$
We want to show that, under our assumptions, $h_{a,b}(c) \leq 0$. We have $3a^2 -8a  +32>0$
for all $a$, and the expression $3a^2 -8a  +32$ is increasing for $a\in [3, 4)$, so
$3a^2 -8a  +32 \leq 3\cdot 16 - 8\cdot 4 +32 = 48,$ and it is sufficient to show that
$$3 b^2c^2 -(2ab +8b^2)c  +48   \leq 0.$$
Since $a\geq 3$, it is sufficient to show that
$$\eta_b(c) := 3 b^2c^2 -(6b +8b^2)c  +48   \leq 0.$$
The vertex point of this parabola is $c_v =\frac{6b +8b^2}{6b^2} = \frac{1}{b} +\frac{4}{3}
\in \left(\frac{4}{3}, \frac{5}{3} \right)$ for $b\geq 3.$ Thus, to prove that  $\eta_b(c) \leq 0$
for all $c, \frac{4}{3} \leq c < \frac{5}{3}$, we need $\eta_b(\frac{4}{3}) \leq 0$ and $\eta_b
(\frac{5}{3}) \leq 0.$ We have $\eta_b(\frac{4}{3}) = - \frac{16}{3}b^2 - 8b +48 \leq 
 - \frac{16}{3}\cdot 9 - 8\cdot 3 +48  = -24 < 0,$ and $\eta_b(\frac{5}{3}) = - 5b^2 - 10b +48 \leq 
 - 45 - 30 +48  = -27 < 0.$ We have proved that $\psi^{\prime\prime}(y)  \geq 0$ for all $ y\in[0, 2].$
The rest of the proof is the same as in the previous case.

Lemma \ref{th:lm1} is
proved.
\end{proof}

We need also the following technical lemma.

\begin{Lemma}
\label{th:lm2}
Let $R_5(z,  \varphi):= \sum_{k=5}^\infty \frac{(-1)^k z^k}{q_2^{k-1} q_3^{k-2} \ldots q_k}$,  
where $q_j >1$ for all $j\geq 2.$
Then 
$$\max_{0 \leq \theta \leq 2\pi}|R_5(q_2e^{i\theta},  \varphi)| \leq 
\frac{q_2q_6}{q_3^3 q_4^2q_5q_6 - q_3^2q_4}.$$ 
\end{Lemma}

\begin{proof}
We have 

$$|R_5(q_2e^{i\theta},  \varphi)| \leq \sum_{k=5}^\infty \frac{q_2^k}{q_2^{k-1} q_3^{k-2} \ldots q_k} = 
\sum_{k=5}^\infty \frac{q_2}{q_3^{k-2} \ldots q_k} =$$ $$\frac{q_2}{q_3^3 q_4^2 q_5} + 
\frac{q_2}{q_3^4 q_4^3 q_5^2 q_6} + \ldots + \frac{q_2}{q_3^{k-2}\ldots q_k} + \ldots $$ 
$$= \frac{q_2}{q_3^3 q_4^2q_5} (1 + \frac{1}{q_3q_4q_5q_6} + 
\frac{1}{q_3^2q_4^2q_5^2q_6^2q_7} + \ldots ) \leq $$
$$ \frac{q_2}{q_3^3 q_4^2q_5} \cdot \frac{1}{1 - \frac{1}{q_3 q_4q_5q_6}} = 
\frac{q_2q_6}{q_3^3 q_4^2q_5q_6 - q_3^2q_4}.  $$
\end{proof}

Let us check that  $$\frac{q_2}{q_3^2q_4}  > \frac{q_2q_6}{q_3^3 q_4^2q_5q_6 - q_3^2q_4},$$
which is equivalent to $$q_3q_4q_5q_6 > q_6 +1.$$ The last inequality obviously holds under 
our assumptions.  Therefore, according to the Rouch\'e theorem, the functions $S_4(z,  \varphi)$ 
and $\varphi (z)$ have the same number of zeros inside the circle $\{ z: |z| < q_2   \}$ counting 
multiplicities. 

It remains to prove that $S_4(z,  \varphi)$  has  zeros in the circle $\{ z: |z| < q_2   \}$.
To do this we need the notion of apolar polynomials and the famous theorem by J.H. Grace.

{\bf Definition 4} (see, for example \cite[Chapter 2, \S 3, p. 59]{PS}).
 Two complex polynomials $P(z) = \sum_{k=0}^n {n \choose k}a_k z^k$ and  
$Q(z) = \sum_{k=0}^n {n \choose k}b_k z^k$ of degree $n$ are called apolar if 
\begin{equation}
\label{apol}
\sum_{k=0}^n (-1)^k {n \choose k} a_k b_{n-k} = 0.
\end{equation}

The following famous theorem due to J.H. Grace states that the complex zeros of two apolar polynomials 
cannot be separated by a straight line or by a circumference.

{\bf Theorem G} (J.H. Grace, see for example \cite[Chapter 2, \$ 3, Problem 145]{PS}). 
{\it Suppose $P$ and $Q$ are two apolar polynomials 
of degree $n \geq 1.$ If all zeros of $P$ lie in a circular region $C,$ then $Q$ has at least one zero in $C.$ 
(A circular region is a closed or open half-plane, disk or exterior of a disk).}

The following lemma was proved in \cite{ngthv2} (see Lemma 2.6 from \cite{ngthv2}). For 
the reader's convenience we give a short proof.

\begin{Lemma} 
\label{th:lm3}
Let $S_4(z,  \varphi) = 1 - z + \frac{1}{q_2}z^2 - \frac{1}{q_2^2q_3}z^3 + 
\frac{1}{q_2^3q_3^2q_4}z^4$ be a polynomial 
and $q_2 \geq 3$. Then $S_4(z,  \varphi)$ has at least one root in the circle $\{z:|z| \leq q_2\}$.
\end{Lemma}

\begin{proof}
We have $$S_4(z,  \varphi) = {4 \choose 0} + {4 \choose 1} (-\frac{1}{4})z + {4 \choose 2} \frac{1}{6q_2}z^2 + 
{4 \choose 3} (-\frac{1}{4q_2^2 q_3})z^3 +$$ $${4 \choose 4} \frac{1}{q_2^3q_3^2q_4}z^4.$$  Let $$Q(z) = 
{4 \choose 2}b_2z^2 + {4 \choose 3}b_3z^3 + {4 \choose 4}z^4.$$ Then the condition for $S_4(z,  \varphi)$ 
and $Q(z)$ to be apolar is the following
$${4 \choose 0} - {4 \choose 1}\left(-\frac{1}{4}\right)b_3 + {4 \choose 2}\frac{1}{6q_2}b_2 = 0.$$

We have $1 + b_3 + \frac{b_2}{q_2} = 0.$  Further, we choose 
$b_3 = \frac{q_2 - 6}{2}$, and, by the apolarity condition, $b_2 = - q_2(1+\frac{q_2 - 6}{2}).$
Hence, we have 
$$Q(z) = - 6 q_2 \left(1 + \frac{q_2 - 6}{2}\right)z^2 + 4 \left(\frac{q_2 - 6}{2}\right) z^3 + z^4 $$
$$ = z^2 \left( - 3q_2 (q_2 - 4) + 2(q_2-6)z + z^2\right).$$

As we can see,  the zeros  of $Q$ are $z_{1} = 0, z_{2} = 0, z_{3} = q_2, z_{4} = - 3(q_2 - 4).$
To show that $z_4$ lies in the circle of radius $q_2,$  we solve the inequality  $|-3(q_2 - 4)| \leq q_2$. 
Hence, we obtain that if $q_2 \geq 3, $ then all the zeros of $Q$ are in the circle $\{z: |z| \leq q_2  \}.$
Therefore, we obtain by the Grace theorem that 
$S_4(z,  \varphi)$ has at least one zero in the circle $\{z:|z| \leq q_2\}.$
\end{proof}

Thus,  $S_4(z,  \varphi)$ has at least one zero in the circle $\{z:|z| \leq q_2\}, $ and, by 
Lemma~\ref{th:lm2} applying to $S_4(z,  \varphi)$, $S_4(z,  \varphi)$ does not have zeros on 
$\{z:|z| = q_2  \}.$ So, the polynomial $S_4(z,  \varphi)$ has at least one zero in the open 
circle $\{z:|z| < q_2\}.$ By Rouch\'e's theorem, the functions $S_4(z,  \varphi)$ 
and $\varphi (z)$ have the same number of zeros inside the circle $\{ z: |z| < q_2   \},$ 
whence $\varphi$ has at least one zero in the open circle $\{z:|z| < q_2\}.$ If $\varphi $
is in the Laguerre-P\'olya class, this zero must be real, and, since the coefficients of $\varphi $
are sign-changing, this zero belongs to $(0, q_2).$ 

Theorem \ref{th:mthm1} is proved.

\section{Proof of Theorem \ref{th:mthm2}}

\begin{proof} 
According to theorem \ref{th:mthm1}, if $\varphi(x) = 1 - x + \sum_{k=2}^\infty  
\frac{ (-1)^k x^k}{q_2^{k-1} q_3^{k-2} \ldots q_{k-1}^2 q_k} \in \mathcal{L-P},$ and 
$ q_3 \geq q_2 \geq 3,  $  then there exists $x_0 \in (0, q_2)$ such that $\varphi(x_0) \leq 0.$ 

For $x \in [0,1]$ we have 
$$ 1 \geq x  > \frac{x^2}{q_2} >  \frac{x^3}{q_2^2 q_3}  > 
\frac{x^4}{q_2^3 q_3^2 q_4 } > \cdots  , $$
whence 
\begin{equation}
\label{mthm1.1}
\varphi(x) >0 \quad \mbox{for all}\quad x\in [0,1].
\end{equation}

Suppose that   $x \in (1, q_2].$  Then we obtain 
\begin{equation}
\label{mthm1.2}  1 <  x  \geq  \frac{x^2}{q_2} > \frac{x^3}{q_2^2 q_3} > 
\cdots >  \frac{  x^k}{q_2^{k-1} q_3^{k-2} \ldots q_{k-1}^2 q_k} > \cdots 
\end{equation}

For an arbitrary  $m \in {\mathbb{N}}$  we have
$\varphi(x) = S_{2m+1}(x, \varphi) + R_{2m+2}(x, \varphi).$
By (\ref{mthm1.2}) and the Leibniz criterion for alternating series, we 
obtain that $R_{2m+2}(x, \varphi) >0$ for all   $x \in (1, q_2],$
or
\begin{equation}
\label{mthm1.3}   \varphi(x) > S_{2m+1}(x, \varphi)\quad \mbox{for all} \quad 
x \in (1, q_2], m \in {\mathbb{N}}.
\end{equation} 

Analogously, 
\begin{equation}
\label{mthm1.4} \varphi (x) < S_{2m}(x, \varphi)\quad \mbox{for all} \quad 
x \in (1, q_2], m \in {\mathbb{N}}.
\end{equation}

We consider $S_3(x, \varphi) = 1 - x + \frac{x^2}{q_2} - \frac{x^3}{a^2b} =: S_3(x).$
Let us use the denotation $a:=q_2, b:= q_3$. Then we obtain
$S_3(x) = 1 - x + \frac{x^2}{a} - \frac{x^3}{a^2b}, b \geq a \geq 3.$
If exists $x_0 \in (0, a)$ such that $\varphi(x_0) \leq 0,$  then $S_3(x_0) \leq 0.$

First, we find the roots of the derivative. We have
$S_3\textprime(x) = -\frac{1}{a^2b}(3x^2 - 2ab x + a^2b).$

We consider the discriminant of the quadratic polynomial $S_3\textprime$: 
$D/4 = a^2b^2 - 3a^2b = a^2b (b - 3).$ Under our assumptions, $b \geq 3,$ so $D/4 \geq 0.$
Thus, the roots are $x_1 = \frac{ab - a \sqrt{b(b-3)}}{3}$
and  $x_2 = \frac{ab + a \sqrt{b(b-3)}}{3}.$ It is easy to check that
if $b \geq a \geq 3, $ then  the following inequalities hold:
$x_1 \in(0, a]$  and $x_2 \geq a.$
Therefore, we can conclude that $x_1$ is the minimal point of $S_3(x)$ in the interval 
$(0, a).$ Now we check if $S_3(x_1) \leq 0.$

After substituting $x_1$ into $S_3(x)$, we obtain the following expression
$$S_3(x_1) = 1 - \frac{ab - a\sqrt{b(b-3)}}{3} + \frac{(ab - a\sqrt{b(b-3)})^2}{9a} - 
\frac{(ab - a\sqrt{b(b-3)})^3}{27a^2b}.$$
We want $S_3(x_1) \leq 0,$  or 
$$27 - 9ab + 9a\sqrt{b(b-3)} + 3ab^2 - 6ab\sqrt{b(b-3)} + 3ab(b-3) - ab^2+$$
$$3ab\sqrt{b(b-3)} - 3ab(b-3) + a(b-3)\sqrt{b(b-3)} \leq 0.$$ 

We rewrite and get
\begin{equation}
\label{ab} \sqrt{b(b-3)}(6a - 2ab) + (27 - 9ab + 2ab^2) \leq 0.
\end{equation}

We have $6a - 2ab = 2a(3 - b) \leq 0$, since $b \geq 3$, under our assumptions. 
Thus, $\sqrt{b(b - 3)}(6a - 2ab) \leq 0.$ Now we consider the expression $27 - 9ab + 
2ab^2 =: y(b)$ as a quadratic function of $b.$ Its discriminant $D = 81a^2 - 216a = 27a(3a - 8)$ 
is positive under our assumption  $a \geq 3.$  The roots of $y(b)$ are
$$b_1 = \frac{9a - 3\sqrt{3a(3a-8)}}{4a} \   \mbox{and} \   
b_2 = \frac{9a + 3\sqrt{3a(3a-8)}}{4a}.$$
It is easy to check that $ b_1 < \frac{9}{4} < 3 \leq b_2.$

For $b \in [3, \frac{9a + 3\sqrt{3a(3a-8)}}{4a}]$ we have $y(b) \leq 0,$
and (\ref{ab}) is fulfilled. Consider now the case $b > \frac{9a + 3\sqrt{3a(3a-8)}}{4a}.$
Then, (\ref{ab}) is equivalent to
$$27 - 9ab + 2ab^2 \leq \sqrt{b(b-3)}(2ab - 6a),$$
or
$$a^2b^2 - 4a^2b - 4ab^2 + 18ab - 27 \geq 0.$$

We rewrite the inequality above in the following way

\begin{equation}
\label{b} b^2a(4 - a) + 2a(2a-9)b + 27 \leq 0.
\end{equation}

We note that the coefficient of $b^2$ is positive, since under our assumptions, $a < 4$.
Then, $D/4 = a^2(2a - 9)^2 - 27a(4-a) = 4a (a-3)^3 \geq 0.$
We obtain the roots of the left-hand side of (\ref{b}): 
$$\beta_1 = \frac{-a(2a-9)-2(a-3)\sqrt{a(a-3)}}{a(4-a)},$$      
$$\beta_2 = \frac{-a(2a-9)+2(a-3)\sqrt{a(a-3)}}{a(4-a)}.$$

Therefore, $b$ should be in $\in (\beta_1, \beta_2)$ for (\ref{b}) to be fulfilled, which is equivalent to
\begin{equation}
\label{b1b2}
\frac{-a(2a-9)-2(a-3)\sqrt{a(a-3)}}{a(4-a)} < b <$$ $$\frac{-a(2a-9)+2(a-3)\sqrt{a(a-3)}}{a(4-a)}.
\end{equation}
Next, we show that the following inequality is fulfilled under our assumptions:
$$ \frac{-a(2a-9)-2(a-3)\sqrt{a(a-3)}}{a(4-a)} \leq \frac{9a + 3\sqrt{3a(3a-8)}}{4a}.$$
It is equivalent to 
$$a^2 < 8(a-3)\sqrt{a(a-3)} + 3(4-a)\sqrt{3a(3a-8)},$$
or
$$a^3 < 64(a-3)^3 + 27(4-a)^2(3a-8) + 48(a-3)(4-a)\sqrt{3(a-3)(3a-8)}.$$
After straightforward calculation we get
$$3(a^3 - 10a^2 + 33a - 36) + (a-3)(4-a) \sqrt{3(a-3)(3a-8)} \geq 0,$$
or
$$3(a-3)^2(a-4) + (a-3)(4-a)\sqrt{3(a-3)(3a-8)} \geq 0. $$
Divided by $(a-3)(4-a):$
$$\sqrt{3(a-3)(3a-8)} \geq 3(a-3).$$
It is simple to verify that the equation above is fulfilled for any $a \geq 3.$ 
Consequently, we obtain the following condition for $b:$
\begin{equation}
\label{ba}
3 \leq b \leq \frac{-a(2a-9)+2(a-3)\sqrt{a(a-3)}}{a(4-a)}.
\end{equation}

Theorem \ref{th:mthm2} is proved.
\end{proof}

{\bf Remark.} {\it From Lemma \ref{lem0}, we have $b (a - 4) + 3 \geq 0,$ or
$b \leq \frac{3}{4-a}.$  It is easy to check  the following inequality:
$$b \leq \frac{-a(2a-9)+2(a-3)\sqrt{a(a-3)}}{a(4-a)} \leq \frac{3}{4-a}.$$}

\section{Proof of Theorem \ref{th:mthm3}}

\begin{proof}
Note that for every $x_0\in (0, q_2)$ and for any $n \in \mathbb{N}: \varphi (x_0) 
\leq S_{2n}(x_0, \varphi)$  (see (\ref{mthm1.4})). Hence, 
if there exists $x_0 \in (0, q_2)$ such that $S_4(x_0, \varphi) \leq 0,$ then 
$\varphi(x_0) \leq 0.$ 

Let
$P(x) := 1 - x + \frac{x^2}{a} - \frac{x^3}{a^2b} + \frac{x^4}{a^3b^2c},$ $a \geq3, b\geq 2, c \geq 3.$

First, if $a \leq c,$ then 
$P(x) \leq 1 - x + \frac{x^2}{a} - \frac{x^3}{a^2b} + \frac{x^4}{a^4b^2} =: 
\widetilde{P}(x).$ Thus, if there exists $x_0 \in (0, a)$ such that $\widetilde{P}(x_0) \leq 0,$ then $P(x_0) \leq 0.$
Next, if $a \geq c,$ then we consider $P(x) = (1 - x) + (\frac{x^2}{a} - \frac{x^3}{a^2b}) + \frac{x^4}{a^4b^2}.$
Since $(\frac{x^2}{a} - \frac{x^3}{a^2b})\textprime_{a} = -\frac{x^2}{a^2} + \frac{2x^3}{a^3b}$ is 
negative for any $x \in (0,a)$, it follows that $P(x)$ decreases monotonically for any $x \in (0,a).$ Therefore, 
$P(x) \leq 1 - x + \frac{x^2}{c} - \frac{x^3}{c^2b} + \frac{x^4}{c^4b^2} =: \widetilde{\widetilde{P}}(x).$ 
Analogously to the previous case, if there exists $x_0 \in (0, a)$ such that $\widetilde{\widetilde{P}}(x_0) 
\leq 0,$ then $P(x_0) \leq 0.$

Thus, we can set $d := \min (a, c),$ and consider the following polynomial
$$T(x) = 1 - x + \frac{x^2}{d} - \frac{x^3}{d^2b} + \frac{x^4}{d^4b^2}.$$

We substitute $x = d\sqrt{b}y,$ then $y = \frac{x}{d\sqrt{b}}, y \in (0, \frac{1}{\sqrt{b}}).$
We have  $Q(y) := T(d\sqrt{b}y) = y^2\bigg( (\frac{1}{y^2} + y^2) - d\sqrt{b} (\frac{1}{y} +y) +db \bigg) .$
Denote by  $w(y) = y + y^{-1},$ then $\bigg( (\frac{1}{y^2} + y^2) - d\sqrt{b} (\frac{1}{y} +y) 
+db \bigg) = w^2 - d\sqrt{b}w + db - 2.$ If we find a point $w_0 \in (\sqrt{b} + \frac{1}{\sqrt{b}}, \infty),$ 
such that $w_0^2 - d\sqrt{b}w_0 + db - 2 \leq 0,$ then we find $y_0,  0 < y_0 < \frac{1}{\sqrt{b}} < 1,$
such that $Q(y_0)\leq 0.$

The vertex of the quadratic function $w^2 - d\sqrt{b}w + db - 2$ is in the point
$w_v= \frac{d\sqrt{b}}{2},$ and,  by our assumptions, $\frac{d\sqrt{b}}{2} \geq \sqrt{b} + 
\frac{1}{\sqrt{b}}.$ Therefore, the existence of a point $w_0 \in (\sqrt{b} + \frac{1}{\sqrt{b}}, \infty),$ 
such that $w_0^2 - d\sqrt{b}w_0 + db - 2 \leq 0,$ is equivalent to the statement
$D= d^2b - 4db + 8  \geq 0.$
Then, $(4-d)db \leq 8,$ or (since $3\leq d <4$)
$$b \leq \frac{8}{d(4-d)}.$$

Theorem \ref{th:mthm3} is proved.
\end{proof}

{\bf Aknowledgement.} The first author is deeply grateful to the Akhiezer Foundation for the 2019 financial support.

\end{document}